\documentclass[11pt,a4paper,draft]{article}
\title{An alternative approach to heavy-traffic limits \\ for finite-pool queues}
\date{\today}
\author{G. Bet}

\usepackage{gianmarcos_definitions} % for Windows

\begin{document}

\maketitle
\abstract{We consider a model for transitory queues in which only a finite number of customers can join. The queue thus operates over a finite time horizon. In this system, also known as the $\Delta_{(i)}/G/1$ queue, the customers decide independently when to join the queue by sampling their arrival time from a common distribution. We prove that, when the queue satisfies a certain heavy-traffic condition and under the additional assumption that the second moment of the service time is finite, the rescaled queue length process converges to a reflected Brownian motion with parabolic drift. Our result holds for general arrival times, thus improving on an earlier result \cite{bet2014heavy} which assumes exponential arrival times.}
\section{Introduction}%
The analysis of transient and time-dependent queueing models is of great relevance for numerous applications, such as call centres \cite{brown2005statistical} and outpatient wards of hospitals where the server operates only over a finite amount of time  \cite{ kim2014callcenter, kim2014choosing}. Besides their practical relevance, these systems provide a substantial mathematical challenge because the standard tools of renewal theory and ergodic theory are unsuited for their study. In other words, the steady-state distribution provides, if any, a poor approximation for the performance measures of transient queueing systems. 

Here we focus on a particular class of transient queues, in which a finite (but large) number $n$ of customers can potentially join. As time passes, fewer customers can join the queue, so that eventually the queue length process will be identically zero and only its time-dependent behavior is of interest. We exploit ideas from the heavy-traffic approximation literature to prove that the queue length process can be approximated by a diffusion consisting of a Brownian motion with parabolic drift, reflected at zero. 

The heavy-traffic approximation approach has been pioneered by Iglehart and Whitt \cite{iglehart1970multipleI} and has since been extended to a wide variety of settings where the time-dependent behavior is of interest, see \cite{glynn1990diffusions} for an excellent overview. Indeed, our result should be contrasted with \cite{iglehart1970multipleI}, where the queue length process is shown to converge to a reflected Brownian motion. The additional parabolic drift captures the effect of the diminishing pool of customers. 

Transient queueing models have been studied lately by Honnappa et al.~\cite{honnappa2015delta, honnappa2014transitory}. However, interest in non-ergodic queues dates back to the pioneering work of Newell on the so-called $M_t/M_t/1$ queue \cite{newell1968queuesI, newell1968queuesII, newell1968queuesIII}.  Later, Keller \cite{keller1982time} rederived Newell's heuristic results by methods of asymptotic expansion of the transition probabilities. Massey \cite{massey1985asymptotic} expanded and formalized these earlier results by using operator techniques. 
More recently, Honnappa, Jain and Ward \cite{honnappa2015delta} introduced the $\DG$ queue as a model for systems in which a finite number of customers can join and/or which operate only over a finite time window. In \cite{honnappa2015delta} the authors prove a Functional Law of Large Numbers (FLLN) and a Functional Central Limit Theorem (FCLT) for the $\DG$ queue under very mild assumptions. In \cite{bet2014heavy}, by exploiting a general \emph{martingale} FCLT from \cite{MarkovProcesses}, it is shown that, when the arrival times are exponentially distributed and under the additional assumption that the queue satisfies a certain heavy-traffic condition, the rescaled queue length process converges in distribution to a reflected Brownian motion with parabolic drift. 

The martingale FCLT is a convenient and powerful tool, but comes at a high cost in terms of computations to verify technical conditions. On the other hand, both the pre-limit and the limit queue length processes are easily characterized through explicit formulas. This suggests that it should possible to prove the convergence result in \cite{bet2014heavy} by using the ``straightforward'' approach to stochastic-process convergence, as detailed e.g.~in \cite{billingsley1999convergence}. As an example, assume a sequence of processes $(\mathcal S_n(\cdot))_{n\geq1}$ and a candidate limit $\mathcal S(\cdot)$ are given. The ``straightforward'' approach consists in proving separately the tightness of the family $(\mathcal S_n(\cdot))_{n\geq1}$, seen as measures on a certain function space, and the convergence of the finite-dimensional distributions, that is, as $n\rightarrow\infty$,
\begin{equation}\label{eq:finite_dimensional_convergence}%
\mathbb P(S_n(t_1)\in A_1,\ldots, S_n(t_k)\in A_k)\rightarrow\mathbb P(S(t_1)\in A_1,\ldots, S(t_k)\in A_k),
\end{equation}%
for each $k\geq1$ and $t_1,\ldots,t_k$. Note that condition \eqref{eq:finite_dimensional_convergence} characterizes the limit process uniquely. 

By exploiting this method, we prove that the queue length process of the $\DG$ queue converges in distribution to a Brownian motion with negative quadratic drift, reflected at zero. In particular, the proof we give is substantially simpler than the one in \cite{bet2014heavy}, requiring only the standard notions of stochastic-process convergence theory \cite{billingsley1999convergence}. This approach has two advantages. First, we impose mild assumptions on the arrival time distribution, thus generalizing \cite{bet2014heavy}, where the arrival times were assumed to be exponentially distributed. Second, as a consequence of our main theorem, several results relating quantities of interest other than the queue length can be deduced. As an example of this, we prove a sample path Little's Law.

The rest of the paper is organized as follows. In Section \ref{sec:model} we describe the $\DG$ model, our assumptions, the processes of interest and state the main result. In Section \ref{sec:proof} we prove the main theorem, by separately proving convergence of the terms appearing in the expression for the queue length process. In Section \ref{sec:sample_path_little_law} we prove a transient version of Little's Law by building on the techniques and results of Section \ref{sec:proof}. In Section \ref{sec:conclusions} we summarize our result and sketch some interesting future research directions.
\section{The model and the main result}\label{sec:model}
We consider a population of $n$ customers. Each customer is assigned a clock $T_i$, with $i=1,\ldots, n$. We assume $(T_i)_{i=1}^{\infty}$ to be a sequence of positive i.i.d.~random variables with common density function $\fT(\cdot)$ and distribution function $\FT(\cdot)$. In particular, $\FT(\cdot)$ is continuous. Customers arrive at a single server with an infinite buffer and are served on a First-Come-First-Served basis. 
The number  of arrivals in $[0,t]$ is then given by
\begin{equation}%
A^n(t) := \sum_{i=1}^n \mathds 1_{\{T_i\leq t\}}.
\end{equation}%
Note that $A^n(t)/n$ is the empirical cumulative distribution function associated with $(T_i)_{i=1}^n$. The service times $(S_i)_{i=1}^{\infty}$ are  i.i.d.~random variables such that $\sigma^2 := \Var(S) <\infty$. The corresponding (rescaled) renewal process is defined as
\begin{equation}\label{eq:service_renewal_process}%
S^n(t) : = \sup \Big\{m\geq1 \mid \sum_{i=1}^m S_i \leq nt\Big\}.
\end{equation}%
We further assume that at time zero the system obeys the \emph{heavy-traffic condition}
\begin{equation}%
\fT(0)= \sup_{t\geq0}\fT(t),
\end{equation}%
and that
\begin{equation}\label{eq:heavy_traffic_condition}%
\E[S]\fT(0) = 1,
\end{equation}%
which can be interpreted as follows. The number of arrivals in the interval $[0,\mathrm d t]$ is approximately $n(\FT(0+\mathrm dt) - \FT(0)) \approx n\fT(0)\mathrm d t$. Consequently, $\lambda_n = n\fT(0)$ represents the \emph{instantaneous} arrival rate in zero. On the other hand, because of the time scaling in \eqref{eq:service_renewal_process}, the service rate is $\mu_n = n/\E[S]$. The heavy-traffic condition is then equivalent to assuming that 
\begin{equation}\label{eq:heavy_traffic_condition_interpretation}%
\frac{\lambda_n}{\mu_n} = 1.
\end{equation}%
More generally, condition \eqref{eq:heavy_traffic_condition_interpretation} could be replaced by $\frac{\lambda_n}{\mu_n} = 1 + \varepsilon_n$, for some $\varepsilon_n\rightarrow0$, but we refrain from doing it here. For a detailed explanation of the condition \eqref{eq:heavy_traffic_condition}, see \cite{bet2014heavy}.
Our main object of interest is the queue length process, defined as
\begin{equation}\label{eq:def_queue_length_process}%
Q^n(t) = A^n(t) - S^n(B^n(t)).
\end{equation}%
Here $B^n(t)$ is a continuous process  that increases at rate $1$ if the server is working, and is constant otherwise.
Note that $A^n(t)$ and $S^n(t)$ are independent as they only depend respectively on $(T_i)_{i\geq1}$ and $(S_i)_{i\geq1}$. They interact through the time-change $t\mapsto B^n(t)$, which depends  on both $(T_i)_{i\geq1}$ and $(S_i)_{i\geq1}$. 
The \emph{diffusion-scaled heavy-traffic} queue length process is defined as 
\begin{equation}\label{eq:diffusion_scaled_queue_length}%
\hat Q ^n(t) := \frac{{Q}^n(tn^{-1/3})}{n^{1/3}}.
\end{equation}%
Recall that the Skorokhod reflection map is the functional %$(\phi,\psi):\mathcal D\rightarrow\mathcal D$ 
defined by 
\begin{align}%
\psi(f)(t) &= - \inf_{ s\leq t} (f(s))^-,\\
\phi(f)(t) &= f(t) + \psi(f)(t).
\end{align}%
We are now able to state our main theorem.
\begin{theorem}[Scaling limit of the queue length process]\label{th:main_theorem}
As $n\rightarrow\infty$,
\begin{equation}\label{eq:main_theorem_reflected_convergence}%
\hat Q ^n(t) \stackrel{\mathrm d}{\rightarrow} \phi(\hat X)(t),\qquad\mathrm{in}~(\mathcal D,J_1),
\end{equation}%
where 
\begin{equation}\label{eq:limiting_free_process}
\hat X(t) = B_1(f_{\sss T}(0) t) - \frac{\sigma}{\E[S]^{3/2}}B_2(t) - \frac{f'_{\sss T}(0)}{2} t^2,
\end{equation}%
and $B_1(\cdot), B_2(\cdot)$ are two independent standard Brownian motions. 
\end{theorem}

\textbf{Notation.}~Here $\mathcal D(\mathbb R) = \mathcal D$ denotes the space of c\`adl\`ag functions with values in $\mathbb R$, that is of functions $f(\cdot):\mathbb R^+ \rightarrow\mathbb R$ which are continuous from the right at every point and such that $\lim_{s\rightarrow t^-} f(s)$ exists for all $t>0$. $\mathcal D$ is endowed with the usual Skorokhod $J_1$ topology. For a sequence of stochastic processes $(X_n)_{n\geq1}$, $X_n \sr{\mathrm d}{\rightarrow} X$ in $(\mathcal D, J_1)$ means that $(X_n)_{n\geq1}$, seen as a sequence of random variables on $\mathcal D$, converges to $X$ in distribution, when $\mathcal D$ is endowed with the $J_1$ topology. Analogously, $X_n \sr{\mathrm d}{\rightarrow} X$ in $(\mathcal D, U)$ means that $(X_n)_{n\geq1}$ converges to $X$ in distribution, uniformly over compact subsets. Recall that, for a sequence $(x_n)_{n\geq1}\subset \mathcal D$, if  $x_n\rightarrow x$ in $(\mathcal D, J_1)$ as $n\rightarrow\infty$, and $x$ is continuous, then $x_n\rightarrow x$ in $(\mathcal D, U)$, see  \cite[p.~124]{billingsley1999convergence}. When dealing with vectors of functions  we make use of the \emph{weak} $J_1$ topology $JW_1$. This coincides with the product topology on $\mathcal D\times \mathcal D\times\cdots \times \mathcal D = \mathcal D^k$. Given two (possibly random) functions, either on the real numbers or on the integers, $f, g$ the notation $f\sim g$ means $\lim_{x\rightarrow\infty} f(x)/g(x) = 1$, where $x\in\mathbb R$ or $x\in\mathbb N$. The notation $f(x) = \oP(g(x))$ means that $f(x)/g(x)\sr{\mathbb{P}}{\rightarrow}0$ as $x\rightarrow\infty$. The notation $f(x) = \Theta(g(x))$ means $f(x) = O(g(x)) $ and $g(x) = O(f(x))$. Finally, $f(x)^+=\max\{0,f(x)\}$ and $f(x)^-=\max\{0,-f(x)\}$ denote the positive and negative part of a function $f(\cdot)$ respectively.

\textbf{The cumulative busy time process.}\qquad We now give an explicit analytical characterization of $B^n(\cdot)$. To this end, we need to introduce several auxiliary processes.

The \emph{cumulative input} process is defined as 
\begin{equation}%
C^n(t) := \sum_{i=1}^{A^n(t)}\frac{S_i}{n}.
\end{equation}%
$C^n(t)$ can be seen as the (rescaled) total amount of work that has entered the queue by time $t$.
Assuming that the server works at speed one, the \emph{net-put process} $N^n(t)$ is defined as
\begin{equation}%
N^n(t) := C^n(t) - t.
\end{equation}%
The \emph{workload} process is then defined as
\begin{equation}%
L^n(t) := \phi(N^n)(t) = N^n(t) - \inf_{s\leq t} (N^n(s))^-.
\end{equation}%
Note that $L^n(t)$ is positive if and only if 
\begin{align}\label{eq:workload_positive_condition}%
C^n(t) &\geq t + \inf_{s\leq t} (N^n(s))^-\nnl
&= t  -\psi(N^n)(t).
\end{align}%
By construction, $\psi(N^n)(t)$ increases (linearly) if and only if the server is idling, and is constant otherwise. In other words, $I^n(t) := \psi(N^n)(t)$ can be interpreted as the cumulative idle time proess. Consequently the term on the right-hand side of \eqref{eq:workload_positive_condition} can be interpreted as the \emph{cumulative busy time} process, and we define it as
\begin{equation}\label{eq:def_cumulative_busy_time}%
B^n(t) := t  -\psi(N^n)(t).
\end{equation}%
Note that $B^n(t)$ increases only if the server is working, and is constant otherwise.  With this definition, \eqref{eq:workload_positive_condition} reads
\begin{equation}%
C^n(t) \geq B^n(t),
\end{equation}%
so that the workload is positive if and only if the cumulative input up to time $t$ is larger than the total time the server has spent processing jobs, and in that case it decreases linearly in time.

\textbf{The queue length process.}\qquad It is more convenient to express $Q^n(t)$ as a reflection of a simpler process $X^n(t)$. We will refer to $X^n(t)$ as the \emph{free process}. To do so, we rewrite \eqref{eq:def_queue_length_process} as
\begin{align}%
Q^n(t) &= \Big(A^n(t) - S^n(B^n(t)) +\frac{B^n(t)}{\E[S]} -\fT(0)t\Big)- \Big(\frac{B^n(t)}{\E[S]} - \fT(0)t\Big)\nnl
&= \Big(A^n(t) - S^n(B^n(t)) +\frac{B^n(t)}{\E[S]} -\fT(0)t\Big) + \fT(0) I^n(t),
\end{align}%
where we used \eqref{eq:heavy_traffic_condition} in the second equality. We define 
\begin{equation}%
X^n(t) = A^n(t) - S^n(B^n(t)) +\frac{B^n(t)}{\E[S]} -\fT(0)t. 
\end{equation}%
We recall that, for a given process $X^n(t)$, the \emph{Skorokhod problem} associated with $X^n(t)$ consists in finding two processes $P(t)$ and $R(t)$ such that $P(t) = X^n(t) + R(t)\geq 0$, $R(t)$ is increasing, and $\int_0^{\infty} X^n(t) \mathrm d R(t) =0$.
Note that $I^n(\cdot)$ is increasing and, by definition of $Q^n(t)$ and $I^n(t)$, 
\begin{equation}%
\int_0^{\infty} Q^n(t)\mathrm d I^n(t) = 0.
\end{equation}%
Then $Q^n(t)$ and $I^n(t)$ are a solution to the Skorokhod problem associated with $X^n(t)$ and, by applying \cite[Proposition 2.2, p.251]{asmussen2003applied} we have the representation 
\begin{equation}%
Q^n(t) = X^n + \psi (X^n)(t) = \phi (X^n) (t),
\end{equation}%
where 
\begin{equation}\label{eq:idle_time_representation}%
\psi (X^n)(t) = -\Big( \frac{B^n(t)}{\E[S]}-\fT(0) t\Big).
\end{equation}%

\textbf{The fluid and diffusive scaling regimes.}\qquad The \emph{fluid-scaled heavy-traffic} queue length process is defined as
\begin{equation}%
\bar{Q}^n (t) := \frac{Q^n(tn^{-1/3})}{n^{2/3}} = n^{1/3}\Big(\frac{A^n(tn^{-1/3})}{n} - \frac{S^n(B^n(tn^{-1/3}))}{n}\Big).
\end{equation}%
Correspondingly, $\bar X ^n(t)$ is defined as
\begin{align}%
\bar{X}^n (t) &:= n^{1/3}\Big(\frac{A^n(tn^{-1/3})}{n} - \frac{S^n(B^n(tn^{-1/3}))}{n}\Big) + n^{1/3}\frac{B^n(tn^{-1/3})}{\E[S]} -\fT(0)t \nnl
&= n^{1/3}\Big(\frac{A^n(n^{-1/3}t)}{n} -\FT(tn^{-1/3})\Big) -  n^{1/3}\Big(\frac{S^n(B^n(tn^{-1/3}))}{n} - \frac{B^n(tn^{-1/3})}{\E[S]}\Big) \nnl
&\quad+ (n^{1/3}\FT(tn^{-1/3}) - \fT(0)t).
\end{align}%
where in the second equality we have added and subtracted $\FT(t)$ in order to rewrite $\bar X^n(t)$. It can be shown through an application of the functional Law of Large Numbers that, as $n\rightarrow\infty$, the fluid-scaled process $\bar Q^n (\cdot)$ converges to a deterministic process $\bar Q(\cdot)$. However, under our heavy-traffic assumption the process $\bar Q(\cdot)$ is identically zero. Because of this, the diffusion-scaled  queue length process can be rewritten as
\begin{equation}%\label{eq:diffusion_scaled_queue_length}%
\hat Q ^n(t) = n^{1/3}\bar{Q}^n(t) = n^{1/3} (\bar{Q}^n(t) - \bar Q(t) ).
\end{equation}%
Accordingly, $\hat X^n(t)$ is defined as
\begin{align}\label{eq:diffusion_scaled_free_process}%
\hat X^n (t) &:= n^{1/3} \bar{X}^n(t)\nnl
&= n^{2/3}\Big(\frac{A^n(tn^{-1/3})}{n} - \FT(tn^{-1/3})\Big) - n^{2/3}\Big(\frac{S^n(B^n(tn^{-1/3}))}{n} - \frac{B^n(tn^{-1/3})}{\mathbb{E}[S]}\Big) \nnl
&\quad+n^{2/3}(\FT(tn^{-1/3}) - f_{\sss T}(0) tn^{-1/3}).
\end{align}%
In order to prove Theorem \ref{th:main_theorem} we will rely on an analogous result for $\hat X^n(\cdot)$. In fact, Theorem \ref{th:main_theorem} is a straightforward consequence of the following:
\begin{theorem}[Scaling limit of the free process]
As $n\rightarrow\infty$,
\begin{equation}\label{eq:main_theorem}%
\hat X^n (t) \stackrel{\mathrm d}{\rightarrow} \hat X(t),\qquad\mathrm{in}~(\mathcal D,J_1),
\end{equation}%
where $\hat X(\cdot)$ is given by
\begin{equation}
\hat X(t) = B_1(f_{\sss T}(0) t) - \frac{\sigma}{\E[S]^{3/2}}B_2(t) - \frac{f'_{\sss T}(0)}{2} t^2,
\end{equation}%
and $B_1(\cdot), B_2(\cdot)$ are two independent standard Brownian motions. 
\end{theorem}

\textbf{The scaling exponents.}\qquad Let us now give an heuristic motivation for the scaling exponents in \eqref{eq:diffusion_scaled_free_process}. Define the general time scaling exponent as $-\alpha$ and the spatial scaling exponent as $\beta$, for some $\alpha, \beta>0$ to be determined, so that $\difX$ is given by
\begin{align}\label{eq:arrival_process_diffusion_scaled_generic_exponents}%
\difX &= n^{\beta}\Big(\frac{A^n(tn^{-\alpha})}{n}-\FT (tn^{-\alpha}) \Big) + n^{\beta}\Big(\frac{S^n(B^n(tn^{-\alpha}))}{n} - \frac{B^n(tn^{-\alpha})}{\mathbb{E}[S]}\Big)\\
 &\quad+ n^{\beta}(F(tn^{-\alpha}) - f_{\sss T}(0) tn^{-\alpha}).\notag 
 %=: \hat A ^n(t) + \hat S^n (t) +  n^{\beta}(\FT(tn^{-\alpha}) - f_{\sss T}(0) tn^{-\alpha}).
\end{align}%
For the deterministic drift to converge to a non-trivial limit it is necessary that $\alpha, \beta$ be such that $2\alpha = \beta$. Indeed, replacing $\FT(t n^{-\alpha})$ with its Taylor expansion up to the second term, we get 
\begin{equation}%
n^{\beta}(\FT(t n^{-\alpha}) - \fT(0) t n^{-\alpha}) =n^{\beta}\Big(\frac{\fT'(0)}{2} t^2 n^{-2\alpha} + o (n^{-2\alpha})\Big).  
\end{equation}%
Moreover, a necessary condition for $\hat A^n (\cdot)$ in \eqref{eq:arrival_process_diffusion_scaled_generic_exponents} to converge to a non-trivial random process is that, for fixed time $t>0$, the variance of $\difA$ be $\Theta(1)$. This is given by
\begin{align}%
\Var(\difA) &= \frac{n^{2\beta}}{n}\Var (\mathds 1_{\{T\leq tn^{-\alpha}\}})\nnl
&= \frac{n^{2\beta}}{n} \mathbb P(T\leq tn^{-\alpha})(1 - \mathbb P(T\leq tn^{-\alpha}))\nnl
&= \frac{n^{2\beta}}{n}(\fT(0)tn^{-\alpha} + o(n^{-\alpha})).
\end{align}%
Then, $\alpha$ and $\beta$ should be such that
\begin{equation}%
\frac{n^{2\beta-\alpha}}{n} = O(1),
\end{equation}%
which, together with $\beta = 2\alpha$, imply that $\alpha = 1/3$ and $\beta = 2/3$.

\textbf{Comparison with known results.}\qquad We conclude by drawing a connection between Theorem \ref{th:main_theorem} and the analogous result in \cite{bet2014heavy}. There, the queue length process is shown to converge to $\phi(X)(t)$, where $X(t) = \sigma B(t) -t^2/2$, where $\sigma^2 = \E[S^2]/\E[S]^3$ and $B(t)$ is a standard Brownian motion. The random process consisting of the sum of two Brownian motions in \eqref{eq:limiting_free_process} is equivalent in distribution to a single Brownian motion with variance equal to 
\begin{equation}%
\fT(0)+\frac{\E[S^2]-\E[S]^2}{\E[S]^3}.
\end{equation}%
By the heavy-traffic condition \eqref{eq:heavy_traffic_condition} this can be simplified  to
\begin{equation}%
\frac{\E[S]^2+\E[S^2]-\E[S]^2}{\E[S]^3} = \frac{\E[S^2]}{\E[S]^3}.
\end{equation}%
Therefore, the two limits are equal in distribution, as expected.

\section{Proof of Theorem \ref{th:main_theorem}}\label{sec:proof}
\subsection{Overview of the proof}
The proof of Theorem \ref{th:main_theorem} proceeds in several steps. These consist in proving convergence of the three terms in \eqref{eq:diffusion_scaled_queue_length} to the respective terms in \eqref{eq:limiting_free_process} separately. The first term in \eqref{eq:diffusion_scaled_queue_length} is the centred and rescaled empirical distribution function of the sequence $(T_i)_{i\geq1}$. Therefore, its convergence to $B_1(\fT(0)t)$ can be seen as a `local Donsker's Theorem', in which the limiting Brownian Bridge is replaced by a Brownian motion. The second term in \eqref{eq:diffusion_scaled_queue_length} is a time-changed, centred and rescaled renewal process and thus converges by a random time-change theorem and the FCLT for renewal processes. The third term also converges trivially to the limiting quadratic drift. Then, 
the convergence \eqref{eq:main_theorem_reflected_convergence} follows immediately from \eqref{eq:main_theorem} by the continuity of the Skorokhod reflection $\phi(x)$ in all $x\in\mathcal C$, the space of real-valued continuous functions, see \cite[Theorem 13.5.1]{StochasticProcess}.
%
%\begin{enumerate}%
%\item Convergence of the drift is automatic and self-explanatory;
%\item Convergence of the first term in \eqref{eq:main_theorem} can be seen as a `local Donsker's Theorem', in which the limiting Brownian Bridge is replaced by a Brownian motion;
%\item Convergence of $n^{1/3}B^n(tn^{-1/3})$ to a certain deterministic limit;
%\item Convergence of $n^{2/3}(S^n(tn^{-1/3})/n - tn^{-1/3}/\E[S])$ to a Brownian limit (while proving this one also easily gets a functional LLN);
%\item Application of a random time change theorem (as found e.g. in \cite{billingsley1999convergence}) to conclude that $S\circ B$ converges to the appropriate limit
%\item Joint convergence of $A, S, B$ happens because of independence and because $B$ converges to a deterministic limit.
%\end{enumerate}%
%%
%
\subsection{A local Donsker's Theorem}
For sake of simplicity, let us define
\begin{equation}\label{eq:arrival_process_diffusion_scaled}%
\difA := n^{2/3}\Big(\frac{A^n(t n^{-1/3})}{n}-\FT (t n^{-1/3}) \Big)
\end{equation}%
and 
\begin{equation}%
\hat A(t) := B_1(\fT(0)t).
\end{equation}%
The goal of this section is to prove the following:
\begin{lemma}[Convergence of the arrival process]\label{lem:local_donsker_theorem}
As $n\rightarrow\infty$,
\begin{equation}\label{eq:claim_local_donsker_theorem}%
\difA\stackrel{\mathrm d}{\rightarrow} \hat A(t),\qquad \mathrm{in}~(\mathcal D, J_1).
\end{equation}%
\end{lemma}

\begin{proof}
The proof proceeds in two steps. First, we prove convergence of the finite-dimensional distributions. This characterizes the limit uniquely. Second, we prove tightness of the family $(\difA)_{n\geq1}$, seen as elements of $\mathcal P ( \mathcal D)$, the space of measures on the Polish space $\mathcal D$ of c\`adl\`ag functions.
By definition, we say that the finite-dimensional distributions of $\hat A^n (\cdot)$ converge to the finite-dimensional distributions of $\hat A(\cdot)$ if, for every $n\in\mathbb N$ and for each choice of $(t_i)_{i=1}^n$ such that $0 < t_1 < t_2 < \ldots < t_n < \infty$ it holds that, as $n\rightarrow\infty$,
\begin{equation}\label{eq:finite_dimensional_convergence_definition}%
(\hat A^n(t_1), ,\ldots,\hat A^n(t_n)) \dconv (\hat A(t_1), \ldots, \hat A(t_n)).
\end{equation}%
For simplicity we shall prove \eqref{eq:finite_dimensional_convergence_definition}  for $t_1<t_2$, the generalization to an arbitrary choice of $(t_i)_{i=1}^n$ being straightforward. We then aim to show that, as $n\rightarrow\infty$,
\begin{equation}%
(\hat A^n (t_1), \hat A^n(t_2) )\dconv (\hat A (t_1), \hat A(t_2) ).
\end{equation}%
Let $\mathcal N (m, v)$ denote a normally distributed random variable with mean $m$ and covariance matrix $v$. Then  $(\hat A (t_1), \hat A(t_2))\sim \mathcal N(m, V_{t_1,t_2})$, with mean $m=(0,0)$ and covariance matrix $V_{t_1,t_2}$ given by
\begin{equation}
V_{t_1,t_2} = \fT(0)\left(
\begin{array}{cc}
t_1 & t_1\wedge t_2\\
t_1\wedge t_2 & t_2
\end{array}\right),
\end{equation}%
where $a\wedge b = \min\{a,b\}$. To show joint convergence, we apply the Cram\'er-Wold device. Given an arbitrary vector $\gamma = (\gamma_1,\gamma_2)\in\mathbb R^2$, we aim to show that, as $n\rightarrow\infty$,
\begin{equation}\label{eq:cramer_wold_device}%
\gamma_1\hat A^n (t_1)+\gamma_2\hat A^n(t_2)\dconv \gamma_1 \hat A(t_1) + \gamma_2 \hat A(t_2).
\end{equation}%
This is done through the following straightforward generalization of the Lindeberg-Feller CLT.
\begin{theorem}[Lindeberg-Feller CLT \cite{klenke2008probability}]\label{th:lindeberg_CLT}%
Let $(X_{n,l})_{l=1}^n$ be an array of random variables such that $\E[X_{n,l}] = 0$ for all $n\geq 1$ and $l\leq n$ and $\sum_{l=1}^n \Var (X_{n,l})\rightarrow1$. Define 
\begin{equation}%
S_n := X_{n,1} + \ldots + X_{n,n}.
\end{equation}%
Assume that the \emph{Lindeberg condition} holds, i.e. for $\varepsilon>0$,
\begin{equation}%
\frac{1}{\Var(S_n)}\sum_{l=1}^n \E[X_{n,l}^2\mathds 1_{\{X_{n,l}^2> \varepsilon^2\Var(S_n)\}}]\rightarrow 0,\qquad n\rightarrow\infty.
\end{equation}%
Then $(S_n)_{n\geq1}$ converges in distribution to a standard normal random variable.
\end{theorem}%
We remark that in the usual formulation of the Lindeberg-Feller CLT it is assumed that $\sum_{l=1}^n \Var (X_{n,l}) = 1$. The proof of the theorem, as presented e.g.~in \cite{klenke2008probability} can be directly generalized to accommodate for the assumption that $\sum_{l=1}^n \Var (X_{n,l})\rightarrow1$. We now take $X_{n,l}$ to be
\begin{equation}%
X_{n,l} = \gamma_1\frac{ \mathds 1_{\{T_l \leq t_1n^{-1/3}\}} - \FT(t_1n^{-1/3})}{n^{1/3}v_{t_1,t_2}} + \gamma_2\frac{ \mathds 1_{\{T_l \leq t_2n^{-1/3}\}} - \FT(t_2n^{-1/3})}{n^{1/3}v_{t_1,t_2}},
\end{equation}%
where $v_{t_1,t_2}$ is a normalizing constant and is given by
\begin{equation}%
v_{t_1,t_2} = \sqrt{\fT(0) (\gamma_1^2 t_1 + \gamma_2^2 t_2 + 2 \gamma_1\gamma_2t_1)}.
\end{equation}%
Recall that $t_1<t_2$ by assumption. In order to deduce the desired convergence \eqref{eq:cramer_wold_device} we are left to check the conditions of Theorem \ref{th:lindeberg_CLT}. Trivially, $ \E[X_{n,l}] = 0$. Moreover, it is possible to explicitly compute $\Var (X_{n,l})$ as follows:
\begin{align}%
\Var (X_{n,l}) &= \frac{\gamma_1 ^2}{n^{2/3}v_{t_1,t_2}^2} (\FT(t_1n^{-1/3}) - \FT(t_1n^{-1/3})^2)\nnl
&\quad+\frac{\gamma_2 ^2}{n^{2/3}v_{t_1,t_2}^2} (\FT(t_2n^{-1/3}) - \FT(t_2n^{-1/3})^2) \nnl
&\quad+\frac{2\gamma_1\gamma_2}{n^{2/3}v_{t_1,t_2}^2}(\FT(t_1n^{-1/3}) - \FT(t_1n^{-1/3})\FT(t_2n^{-1/3}))\nnl
&= \frac{\fT(0)}{v_{t_1,t_2}^2} \Big(\frac{\gamma_1^2}{n} t_1 + \frac{\gamma_2^2}{n} t_2 + 2 \frac{\gamma_1\gamma_2}{n}t_1\Big) + O(n^{-4/3}),
\end{align}%
where in the second equality the distribution function $\FT(\cdot)$ was Taylor expanded. In particular,
\begin{equation}
\sum_{l=1}^n \Var(X_{n,l}) = 1 + O(n^{-1/3}).
\end{equation}
The \emph{Lindeberg condition} is also satisfied, since
%
%. Indeed,
%%
%\begin{equation}%
%\vert X_{n,l}\vert ^2 \leq \frac{(\gamma_1+\gamma_2)^2}{n^{2/3}v_{t_1,t_2}^2},
%\end{equation}%
%%
%and $\Var (S_n) = \sum_{l=1}^n\Var (X_{n,l}) = \Theta(1)$.
%
\begin{align}%
\sum_{l=1}^n &\frac{1}{n^{2/3}v_{t_1,t_2}}\E[(\mathds 1_{\{T_i\leq t_1n^{-1/3}\}}-\FT(t_1 n^{-1/3}))^2\mathds 1_{\{(\mathds 1_{\{T_i\leq t_1n^{-1/3}\}}-\FT(t_1 n^{-1/3}))\geq \varepsilon n^{1/3}\}}]\nnl
&=\frac{n^{1/3}}{v_{t_1,t_2}}\E[(\mathds 1_{\{T_1\leq t_1n^{-1/3}\}}-\FT(t_1 n^{-1/3}))^2\mathds 1_{\{(\mathds 1_{\{T_1\leq t_1n^{-1/3}\}}-\FT(t_1 n^{-1/3}))\geq \varepsilon n^{1/3}\}}]\nnl
&\leq \frac{n^{1/3}}{v_{t_1,t_2}}\sqrt{\E[(\mathds 1_{\{T_1\leq t_1n^{-1/3}\}}-\FT(t_1 n^{-1/3}))^4]}\sqrt{ \mathbb P (\mathds 1_{\{T_1\leq t_1n^{-1/3}\}}-\FT(t_1 n^{-1/3})\geq \varepsilon n^{1/3})},
\end{align}%
by the Cauchy-Schwartz inequality. The first term is of the order $O(n^{-1/3})$, while the second is identically zero for $n$ large enough. 

By Theorem \ref{th:lindeberg_CLT}, 
\begin{equation}\label{eq:cramer_wold_device_applied}%
\frac{1}{v_{t_1,t_2}}(\gamma_1,\gamma_2)\cdot(\hat{A}^n(t_1),\hat{A}^n(t_2))^{\mathrm t}\dconv \mathcal N(0,1),
\end{equation}%
where $\cdot$ denotes the usual scalar product and $q^{\mathrm t}$ denotes the transpose of a vector $q$. However, since
\begin{equation}%
(\gamma_1,\gamma_2)\cdot V_{t_1,t_2}\cdot (\gamma_1,\gamma_2)^{\mathrm t} = v_{t_1,t_2}^2,
\end{equation}%
then
\begin{equation}%
\mathcal N(0,1)\sr{\mathrm d}{=} \frac{1}{v_{t_1,t_2}}(\gamma_1,\gamma_2)\cdot\mathcal N ((0,0), V_{t_1,t_2}),
\end{equation}%
and this together with \eqref{eq:cramer_wold_device_applied} implies \eqref{eq:cramer_wold_device}. By an application of the Cram\'er-Wold device, joint convergence follows.

The last step of the proof is to show that $(\hat{A}^n(\cdot))_{n=1}^{\infty}$ is a tight family of random variables on $\mathcal D$. By \cite[Theorem 13.5]{billingsley1999convergence}, in particular equation (13.14), it is enough for $(\hat{A}^n(\cdot))_{n=1}^{\infty}$ to satisfy the following condition. For every $T>0$,
\begin{equation}\label{eq:tightness_criterion}%
\E[\vert  \difA - \hat A^n(t_1)\vert^2\vert \hat A^n(t_2) - \difA\vert^2] \leq  (f_{\textup{inc}}(t_2) - f_{\textup{inc}}(t_1))^2,
\end{equation}%
for $0\leq t_1\leq t\leq t_2\leq T$ and $f_{\textup{inc}}(\cdot)$ is a non-decreasing function. Checking \eqref{eq:tightness_criterion} amounts to computing the mean appearing on the left side of the equation. Define
\begin{align}%
%\left\{ 
%\begin{array}{l}
p_1 &:= \FT(tn^{-1/3}) - \FT (t_1n^{-1/3}), \nnl
p_2 &:= \FT(t_2n^{-1/3}) - \FT(tn^{-1/3}). 
%
%p_3 := 1 - (\FT(t_2n^{-1/3}) - \FT(t_2n^{-1/3}))
%\end{array}
%\right.
\end{align}%
Define also
\begin{equation}%
\alpha_i:=\left\{ \begin{array}{ll}
1-p_1, & \mathrm{if}~\frac{T_i}{n^{1/3}}\in(t_1,t],\\
-p_1, & \mathrm{if}~\frac{T_i}{n^{1/3}}\nin(t_1,t],
\end{array}\right.
\end{equation}%
and 
\begin{equation}%
\beta_i:=\left\{ \begin{array}{ll}
1-p_2, & \mathrm{if}~\frac{T_i}{n^{1/3}}\in(t,t_2],\\
-p_2, & \mathrm{if}~\frac{T_i}{n^{1/3}}\nin(t,t_2],
\end{array}\right.
\end{equation}%
where we have omitted dependencies on $n$ to avoid cumbersome notation. Note that $\E[\alpha_1] = \E[\beta_1] = 0$. With the help of these definitions, \eqref{eq:tightness_criterion} can be rewritten in the following form:
\begin{equation}\label{eq:tightness_criterion_rewritten}%
\E\Big[\big(\sum_{i=1}^n\alpha_i\big)^2\big(\sum_{i=1}^n\beta_i\big)^2\Big]\leq n^{4/3}(f_{\textup{inc}}(t_2) - f_{\textup{inc}}(t_1))^2.
\end{equation}%
We will take $f_{\mathrm {inc}} (t) = \sqrt{K} t$ for a certain constant $K>0$. By definition $\alpha_i$ (resp. $\beta_i$) is independent from $\alpha_j$ and $\beta_j$ for $j\neq i$, so that the left side of  \eqref{eq:tightness_criterion_rewritten} can be simplified as 
\begin{align}%
n\E[\alpha_1^2\beta_1^2] + n(n-1) \E[\alpha_1^2]\E[\beta_2^2] + 2n(n-1)\E[\alpha_1\beta_1]\E[\alpha_2\beta_2].
\end{align}%
The first term $n\E[\alpha_1^2\beta_1^2]$ is of lower order, so we focus on the remaining two. A simple computation gives
\begin{align}%
\E[\alpha_1^2] &= p_1(1-p_1)\leq p_1,\nnl
\E[\beta_1^2] &= p_2(1-p_2)\leq p_2,\nnl
\E[\alpha_1\beta_2] &= -p_1p_2,
\end{align}%
so that, since $p_1\leq(p_1 + p_2)$ and $p_2\leq (p_1 + p_2)$,
\begin{align}%
\E\Big[\big(\sum_{i=1}^n\alpha_i\big)^2\big(\sum_{i=1}^n\beta_i\big)^2\Big] &\leq C_0 n^2 p_1p_2 \leq C_0 n^2 (p_2 + p_1)^2\nnl
&= C_0 n^2 (\FT(t_2 n^{-1/3}) - \FT(t_1 n^{-1/3}))^2 \nnl
&\leq C_1 n^{4/3} \fT(0) (t_2 - t_1)^2,
\end{align}%
for a sufficiently large $C_1>0$. Therefore, we have verified \eqref{eq:tightness_criterion_rewritten} with $f_{\mathrm {inc}} (t) = \sqrt{C_1 \fT(0)}t$.
\end{proof}
\subsection{A functional CLT for renewal processes}
We define
\begin{equation}\label{eq:service_process_diffusion_scaled}%
\hat S^n(t) := n^{2/3}\Big(\frac{S^n(tn^{-1/3})}{n} - \frac{1}{\E[S]}tn^{-1/3}\Big)
\end{equation}%
and 
\begin{equation}%
\hat S(t) := \frac{\sigma}{\E[S]^{3/2}}B_2(t),
\end{equation}%
where $\sigma^2$ is the variance of $S$.
The goal of this section is then to prove the following:
\begin{lemma}[Convergence of the service process]\label{lem:clt_renewal_processes}
As $n\rightarrow\infty$,
\begin{equation}\label{eq:claim_clt_renewal_processes}%
\hat S^n(t)\stackrel{\mathrm d}{\rightarrow} \hat S(t), \qquad\mathrm{in}~(\mathcal D, J_1).
\end{equation}%
\end{lemma}
\begin{proof}
Note that $S^n(tn^{-1/3}) = S^{n^{2/3}}(t)$. Moreover,
\begin{equation}%
n^{2/3}\Big(\frac{S^n(tn^{-1/3})}{n} - \frac{1}{\E[S]}tn^{-1/3}\Big) =
\frac{S^{n^{2/3}}(t)- \E[S]^{-1}tn^{2/3}}{n^{1/3}}.
\end{equation}%
Therefore, the claim \eqref{eq:claim_clt_renewal_processes} can be proven by directly applying a FCLT for renewal processes, see e.g.~\cite[Theorem 14.6]{billingsley1999convergence}. 
\end{proof}
\subsection{Convergence of the cumulative busy time}
In this section we exploit Lemma \ref{lem:clt_renewal_processes} and the random time change theorem to prove that the rescaled service process in \eqref{eq:diffusion_scaled_queue_length} converges.
First, we prove some scaling limits for the arrival process. Define the fluid-scaled arrival process as
\begin{equation}%
\bar A^n(t) := \frac{A^n(tn^{-1/3})}{n^{2/3}}.
\end{equation}%
The following straightforward generalization of the Chebyshev inequality is useful when proving the strong Law of Large Numbers:
\begin{lemma}[Generalized Chebyshev inequality]\label{lem:generalized_chebyshev_inequality}
For any $p=1,2,\ldots$, and any random variable $X$ such that $\E[\vert X\vert^p]<\infty$,
\begin{equation}%
\mathbb P(\vert X\vert \geq \varepsilon) \leq \frac{\E[\vert X\vert ^p]}{\varepsilon^p}.
\end{equation}%
\end{lemma}%
By using Lemma \ref{lem:generalized_chebyshev_inequality} together with the Borel-Cantelli lemma, we can prove the following:
\begin{lemma}[LLN for the arrival process]\label{lem:LLN_arrival_process}%
As $n\rightarrow\infty$,
\begin{equation}\label{eq:arrival_process_pointwise_local_LLN}%
\Big\vert \bar A^n(t) - \fT(0)t \Big\vert\asconv0,
\end{equation}%
for fixed $t\geq0$.
\end{lemma}%
\begin{proof}
First, we rewrite
\begin{equation}%
\bar A ^n(t) -\fT(0)t = \frac{1}{n}\sum_{i=1}^n (n^{1/3}\mathds 1_{\{T_i\leq tn^{-1/3}\}} - n^{1/3}\FT(tn^{-1/3})) =: \frac{1}{n}\sum_{i=1}^n Y_i.
\end{equation}%
In order to apply the Borel-Cantelli lemma, we compute
\begin{align}%
\mathbb P\Big(\vert\sum_{i=1}^n Y_i\vert \geq \varepsilon n\Big) & \leq \frac{\E[\vert \sum_{i=1}^n Y_i\vert^4]}{n^4\varepsilon^4} \nnl
&= \frac{n \E[\vert Y_1 \vert ^4] + 3n(n-1)\E[\vert Y_1\vert^2]^2}{n^4\varepsilon^4}.
\end{align}%
It is immediate to see that the leading orders of the expectation values are
\begin{align}%
\E[\vert Y_1\vert^4] &= O(n^{4/3}\mathbb P(T_i\leq tn^{-1/3}))= O(tn),\nnl
\E[\vert Y_1\vert^2] &= O(n^{2/3}\mathbb P(T_i\leq tn^{-1/3})) = O(tn^{1/3}).
\end{align}%
We conclude that, for a large constant $C_1>0$,
\begin{equation}%
\mathbb P\Big(\vert\sum_{i=1}^n Y_i\vert \geq \varepsilon n\Big) \leq C_1 \frac{tn^2 + 3tn^{8/3}}{n^4\varepsilon^4}.
\end{equation}%
Define the event $\mathcal A := \{\vert\sum_{i=1}^n Y_i\vert \geq \varepsilon n  ~\mathrm{for~infinitely~many~}n\}$. Since
\begin{equation}%
\sum_{n=1}^{\infty} \mathbb P\Big(\vert\sum_{i=1}^n Y_i\vert \geq \varepsilon n\Big) \leq C_1 \sum_{n=1}^{\infty} \frac{tn^2 + 3tn^{8/3}}{n^4\varepsilon^4} \leq C_2 \sum_{n=1}^{\infty} \frac{1}{n^{4/3}\varepsilon^4}<\infty,
\end{equation}%
for some large constant $C_2>0$, by the Borel-Cantelli lemma,
\begin{equation}%
\mathbb P(\mathcal A)=0.
\end{equation}%
Since $\varepsilon>0$ is arbitrary, this concludes the proof of \eqref{eq:arrival_process_pointwise_local_LLN}.
%By Lemma \ref{lem:local_donsker_theorem}, for fixed $t\geq0$,
%%
%\begin{equation}%
%n^{2/3}\Big\vert \frac{A^n(t n^{-1/3})}{n}-\FT (t n^{-1/3})\Big\vert\dconv \vert Y\vert,
%\end{equation}%
%%
%where $Y\sim \mathcal N(0, \fT(0)t)$, since the function $x\mapsto \vert x\vert$ is continuous on $\mathbb R$. In particular we have that, as $n\rightarrow\infty$,
%%
%\begin{equation}\label{eq:arrival_process_pointwise_local_LLN_proof}%
%\Big\vert \bar A^n(t) - n^{1/3}\FT(t n^{-1/3})\Big\vert \Pconv 0.
%\end{equation}%
%%
%By the triangle inequality,
%%
%\begin{align}%
%\vert \bar A^n(t) - \fT(0)t \vert &\leq \vert \bar A^n(t) - n^{1/3}\FT(t n^{-1/3})\vert + \vert n^{1/3}\FT(t n^{-1/3})-\fT(0)t \vert.
%\end{align}%
%%
%By \eqref{eq:arrival_process_pointwise_local_LLN_proof} the first term on the right side converges to zero in probability, and the second also converges to zero, as can be seen by Taylor expanding $\FT(tn^{-1/3})$. 
\end{proof}%
We are now interested in obtaining a Glivenko-Cantelli-type theorem which extends the convergence \eqref{eq:arrival_process_pointwise_local_LLN} to uniform convergence over compact subsets of the positive half-line. This is summarized in the following lemma.
\begin{lemma}[Glivenko-Cantelli Theorem for the arrival process]\label{th:glivenko-cantelli_arrival_process}%
As $n\rightarrow\infty$,
\begin{equation}\label{eq:glivenko-cantelli_arrival_process}%
\bar A^n(t) \asconv \fT(0)t,\qquad \mathrm{in}~(\mathcal D,U).
\end{equation}%
Consequently, as $n\rightarrow\infty$,
\begin{equation}\label{eq:fCLT_cumulative_input_process}%
n^{1/3}C^n(t n^{-1/3}) \asconv t \qquad\mathrm{in}~(\mathcal D, U).
\end{equation}%
\end{lemma}%
\begin{proof}%
Let $T>0$ be arbitrary. The claim \eqref{eq:glivenko-cantelli_arrival_process} is then equivalent to
\begin{equation}%
\lim_{n\rightarrow\infty}\sup_{t\leq T}\Big\vert \bar A^n(t) - \fT(0) t\Big\vert = 0,\qquad\mathrm{a.s.}
\end{equation}%
Let $N$ be a large but arbitrary natural number and define 
\begin{equation}%
t_j := \frac{1}{\fT(0)}\frac{j}{N}T,\qquad j=1,\ldots, N,
\end{equation}%
so that $\fT(0)t_j = \frac{j}{N}T$. The idea is that both $A^n(t)$ and $\fT(0)t$ are increasing, so for $t\in(t_{j-1},t_j)$ the difference of the two can be bounded by their values in $t_{j-1}$ and $t_j$. Then, we have convergence because of Lemma \ref{lem:LLN_arrival_process} and because $N$ is fixed. Formally, define the error as
\begin{equation}%
E_{n,N}:= \max_{j=1,\ldots,N}(\vert A^n(t_jn^{-1/3})/n^{2/3} - \fT(0)t_j\vert + \vert A^n(t_j^-n^{-1/3})/n^{2/3} - \fT(0)t_j^-\vert).
\end{equation}%
For $t\in(t_{j-1},t_j)$ we upper bound $\bar A^n(t)$ as follows
\begin{equation}%
\bar A^n(t)\leq \bar A^n(t_j^-)\leq \fT(0) t_j^- + E_{n,N} \leq \fT(0)t + E_{n,N} + \frac{T}{N},
\end{equation}%
where in the last inequality we used the bound $\vert \fT(0)t_{j}-\fT(0)t_{j-1}\vert \leq \frac{T}{N}$. Analogously, for the lower bound
\begin{equation}%
\bar A^n(t) \geq \bar A^n(t_{j-1})\geq \fT(0) t_{j-1} - E_{n,N} \geq \fT(0)t - E_{n,N} - \frac{T}{N}.
\end{equation}%
Summarizing the two bounds, since $E_{n,N}$ and $T/N$ do not depend on the choice of the sequence $(t_j)_{j=1}^N$,
\begin{equation}%
\sup_{t\leq T}\vert \bar A^n(t) - \fT(0) t \vert \leq E_{n,N} + \frac{T}{N}.
\end{equation}%
Since $N$ is fixed, almost surely
\begin{equation}%
\lim_{n\rightarrow \infty}E_{n,N} = 0,
\end{equation}%
by Lemma \ref{lem:LLN_arrival_process}. Letting $N\rightarrow\infty$, we obtain \eqref{eq:glivenko-cantelli_arrival_process}.

The convergence \eqref{eq:fCLT_cumulative_input_process} follows from \eqref{eq:glivenko-cantelli_arrival_process}. Indeed, by the  functional strong Law of Large Numbers \cite[Theorem 5.10]{chen2001fundamentals} we have that 
\begin{equation}%
\sum_{i=1}^{tn^{2/3}} \frac{S_i}{n^{2/3}}\asconv \E[S]t \qquad\mathrm{in}~(\mathcal D, U).
\end{equation}%
Since $\bar A^n(t)$ converges to a deterministic limit, we also have the joint convergence
\begin{equation}%
\Big(\sum_{i=1}^{tn^{2/3}} \frac{S_i}{n^{2/3}},\bar A^n(t)\Big)\asconv (\mathbb E[S]t, \fT(0)t),\qquad \mathrm{in}~(\mathcal D^2, WJ_1).
\end{equation}%
Note that $A^n(t)$ is non-decreasing. Then, by a time-change theorem \cite[Lemma p.151]{billingsley1999convergence},
\begin{equation}%
\sum_{i=1}^{A^n(tn^{-1/3})}\mkern-9mu \frac{S_i}{n^{2/3}} \asconv\E[S]\fT(0) t\qquad\mathrm{in}~(\mathcal D,U).
\end{equation}%
Recall that convergence in $(\mathcal D, J_1)$ to a continuous function implies convergence in $(\mathcal D, U)$. Moreover,  $\E[S]\fT(0) = 1$ by the heavy-traffic condition \eqref{eq:heavy_traffic_condition}, and this concludes the proof of \eqref{eq:fCLT_cumulative_input_process}.
\end{proof}%
Since $t\mapsto \fT(0) t$ is not a proper distribution function, Theorem \ref{th:glivenko-cantelli_arrival_process} should also be interpreted as a \emph{local} version of the usual Glivenko-Cantelli Theorem.
Let us now define the fluid-scaled cumulative busy time process as
\begin{equation}%
\bar B^n (t) := n^{1/3}B^n(tn^{-1/3}).
\end{equation}%
We are able to prove the following lemma:
\begin{lemma}[Convergence of the time-changed service process]
With assumptions as above, as $n\rightarrow\infty$,
\begin{equation}%
\bar B^n(t)\asconv t,\qquad \mathrm{in}~(\mathcal D, U),
\end{equation}%
\end{lemma}
\begin{proof}
Recall that $B^n$ can be rewritten as
\begin{equation}%
B^n(t) = t +  \Psi(N^n)(t)= t +\inf_{s\leq t} (C^n(s) - s)^-.
\end{equation}%
By Lemma \ref{th:glivenko-cantelli_arrival_process}, $n^{1/3}(C^n(t n^{-1/3}) - t n^{-1/3})\asconv 0$ in $(\mathcal D, U)$. Moreover, the null function is a continuity point of $\Psi(\cdot)$ with probability one \cite[Lemma 13.4.1]{StochasticProcess}. The claim then follows from the Continuous Mapping Theorem \cite[Theorem 3.4.3]{StochasticProcess}.
\end{proof}

\subsection{Proof of Theorem \ref{th:main_theorem}}
Since $\bar B^n(\cdot)$ converges to a deterministic limit, we have
\begin{equation}%
(\hat A^n(t), \hat S^n(t), \bar B^n(t))\dconv (\hat A(t), \hat S(t), t),\qquad\mathrm{in}~(\mathcal D^3, WJ_1).
\end{equation}%
Note also that $\hat A^n (\cdot)$ and $\hat S^n(\cdot)$ are independent processes, so that $\hat A(\cdot)$ and $\hat S(\cdot)$ are also independent. Applying the random time-change theorem \cite[Lemma p.151]{billingsley1999convergence}, we get
\begin{equation}%
(\hat A^n(t), \hat S^n(\bar B^n(t)) )\dconv (\hat A(t), \hat S(t)),\qquad\mathrm{in}~(\mathcal D^2, WJ_1).
\end{equation}%
Since the limit points are continuous, by \cite[Theorem 4.1]{whitt1980some} addition is also continuous, so that
\begin{equation}%
\hat A^n(t) - \hat S^n(t) + n^{2/3}(\FT(tn^{-1/3}) - \fT(0)tn^{-1/3})\dconv \hat A(t) - \hat S(t) -\frac{\fT'(0)}{2}t^2,\quad\mathrm{in}~(\mathcal{D}, J_1),
\end{equation}%
which is the first claim \eqref{eq:main_theorem}. By \cite[Theorem 13.5.1]{StochasticProcess}, the reflection map $\phi(\cdot)$ is continuous when $\mathcal D$ is endowed with the $J_1$ topology, from which the second claim \eqref{eq:main_theorem_reflected_convergence} follows. 
\qed
\section{Sample path Little's Law}\label{sec:sample_path_little_law}
In this section we apply the ideas and results from the previous sections to  derive a sample path version of Little's Law. The standard formulation of Little's Law relates the expected waiting time $\E[W]$,  to the expected queue length $\mathbb E[L_q]$ as $\E[L_q] = \lambda \E[W]$, where $\lambda$ is the rate at which customers arrive. We will work instead with the \emph{virtual} waiting time $W^n(t)$, defined as
\begin{equation}\label{eq:virtual_waiting_time}
W^n(t) := C^n(t) - B^n(t).
\end{equation}%
Accordingly, we define the diffusion-scaled virtual waiting time as 
\begin{equation}%
\hat W^n(t) := n^{2/3}\Big(C^n(tn^{-1/3}) -  B^n(tn^{-1/3})\Big) = n^{1/3}\Big(\sum_{i=1}^{A^n(tn^{-1/3})}\frac{S_i}{n^{2/3}} -  \bar B^n(t)\Big).
\end{equation}%
First, we rewrite the expression for $\hat W^n(t)$ as
\begin{align}\label{eq:virtual_waiting_time_diffusion_scaled}%
\hat W^n(t) &= n^{1/3}\Big(\sum_{i=1}^{\bar A^n(t)n^{2/3}}\frac{S_i}{n^{2/3}} - \E[S] \bar A^n(t)\Big)+ n^{1/3}\E[S](  \bar A^n(t) -  n^{1/3} \FT(tn^{-1/3})) \nnl
&\quad  + n^{1/3}\E[S](\FT(tn^{-1/3})-\fT(0)t) + n^{1/3}\E[S](\fT(0)t-\bar B^n(t)/\E[S]).
\end{align}%
By \eqref{eq:idle_time_representation}, $n^{1/3}(\fT(0)t- \bar B^n(t) /\E[S]) = \psi (\hat X^n)(t)$, so that \eqref{eq:virtual_waiting_time_diffusion_scaled} can be further simplified as
\begin{align}\label{eq:virtual_waiting_time_diffusion_scaled_rewritten}%
\hat W^n (t) &= \E[S]\hat Q^n(t)+ n^{1/3}\Big(\sum_{i=1}^{\bar A^n(t)n^{2/3}}\frac{S_i}{n^{2/3}} - \E[S]\bar A^n(t)\Big) + \E[S]\hat S^n(\bar B^n(t)).
\end{align}%
We now focus on the second and third terms in \eqref{eq:virtual_waiting_time_diffusion_scaled_rewritten}. Let us ignore the time change $t\mapsto \bar A^n(t)$ and $t\mapsto\bar B^n(t)$ for the moment. Then, the second and third terms in \eqref{eq:virtual_waiting_time_diffusion_scaled_rewritten} represent the difference between the diffusion-scaled partial sums and the (negative) diffusion-scaled counting process associated with the sequence of random variables $(S_i)_{i\geq1}$. These converge to the same limiting Brownian motion, so that their contribution to $\hat W^n(t)$ vanishes in the limit. We now aim to make this reasoning rigorous.
\begin{theorem}[Diffusion sample path Little's Law]\label{th:diffusion_sample_path_little_law}
As $n\rightarrow\infty$,
\begin{equation}%
\hat W^n(t) \dconv \hat W(t),\qquad\mathrm{in}~(\mathcal{D}, J_1),
\end{equation}%
where 
\begin{equation}%
\hat W(t) := \E[S]\hat Q(t).
\end{equation}%
\end{theorem}%
\begin{proof}
Define the diffusion-scaled partial sum process as
\begin{equation}%
\hat P^n(t) = n^{1/3} \Big(\sum_{i=1}^{tn^{2/3}}\frac{S_i}{n^{2/3}}-\E[S]\bar A^n(t)\Big).
\end{equation}%
By \cite[Theorem 7.3.2]{StochasticProcess}, $\hat P^n(\cdot)$ and $\hat S^n(\cdot)$ jointly converge as follows:
\begin{equation}%
(\hat P^n(t), \hat S^n(t)) \dconv (- \E[S]\hat S(\E[S]t), \hat S(t)),\qquad\mathrm{in}~(\mathcal D^2, WJ_1),
\end{equation}%
where $\hat S^n(t)$ is the same as in \eqref{eq:claim_clt_renewal_processes}. Since $\hat A^n(t)$ is independent from $\hat P^n(t)$ and $\hat S^n(t)$,
\begin{equation}%
(\hat A^n(t), \hat P^n(t), \hat S^n(t)) \dconv (\hat A(t), - \E[S]\hat S(\E[S]t), \hat S(t)),\qquad\mathrm{in}~(\mathcal D^3, WJ_1).
\end{equation}%
Moreover, since $\bar A^n(\cdot)$ and $\bar B^n(\cdot)$ converge to deterministic limits,  by \cite[Theorem 11.4.5]{StochasticProcess} the above convergence can be strengthened to 
\begin{equation}%
(\hat A^n(t),\hat P^n(t), \hat S^n(t),\bar A^n(t),\bar B^n(t)) \dconv (\hat A(t), - \E[S]\hat S(\E[S]t), \hat S(t), \fT(0)t, t),\quad\mathrm{in}~(\mathcal D^4, WJ_1).
\end{equation}%
It follows that 
\begin{equation}\label{eq:little_law_proof_joint_convergence}%
(\hat A^n(t),\hat P^n(\bar A^n(t)), \E[S]\hat S^n(\bar B^n(t))) \dconv (\hat A(t),- \E[S]\hat S(t),\E[S]\hat S(t)),\qquad\mathrm{in}~(\mathcal D^3, WJ_1),
\end{equation}%
by the heavy-traffic assumption \eqref{eq:heavy_traffic_condition}. The limit functions are  continuous with probability one, and thus their sums converge to the sums of the limits. This observation, together with the Continuous Mapping Theorem and \eqref{eq:little_law_proof_joint_convergence} imply that
\begin{equation}%
\E[S]\hat Q^n(t)+ \hat P^n(\bar A^n(t)) + \E[S]\hat S^n(\bar B^n(t))\dconv \hat Q(t),\qquad\mathrm{in}~(\mathcal D, J_1),
\end{equation}%
as $n\rightarrow\infty$, as desired.
%%
%\begin{equation}%
% d_{J_1} (-\hat P^n(\bar A^n(\cdot)), \E[S]\hat S^n(\bar B^n(\cdot)))\dconv 0,
%\end{equation}% 
%%
%where $d_{J_1}(\cdot, \cdot)$ is the metric induced by the Skorokhod norm. Moreover, since
%%
%\begin{equation}%
%d_{J_1}(\hat W^n(\cdot), \E[S]\hat Q^n(\cdot)) =  d_{J_1} (-\hat P^n(\bar A^n(\cdot)), \E[S]\hat S^n(\bar B^n(\cdot))),
%\end{equation}%
%%
%by the convergence-together theorem \cite[Theorem 11.4.7]{StochasticProcess},
%%
%\begin{equation}%
%(\hat W^n(t), \E[S]\difQ)\dconv (\E[S] \hat Q(t), \E[S] \hat Q(t))\qquad\mathrm{in}~(\mathcal D^2, WJ_1),
%\end{equation}%
%%
%from which the claim follows.
\end{proof}%

By our assumptions, $\E[S] = 1/\fT(0)$ so that we retrieve the usual form of Little's Law as
\begin{equation}%
\hat Q(t) = \fT(0)\hat W(t).
\end{equation}%
Note that $\fT(0) = \lambda$ when $T$ is exponentially distributed with mean $1/\lambda$. 

Theorem \ref{th:diffusion_sample_path_little_law} should be contrasted with the analogous result in \cite[Proposition 4]{honnappa2015delta}. There, an extra diffusion term appears. This term is a function of the fluid limit of the queue length process. However, in our setting, this limit is the zero process, as can be seen in \eqref{eq:diffusion_scaled_queue_length}, where no centering is needed and thus the term disappears.
\section{Conclusions}\label{sec:conclusions}
While the $\DG$ queue originated as a simple model for the study of general time-inhomogeneous queueing systems, it has very recently gained much attention since it represents the standard model for queues in which only a finite number of customers request for service \cite{honnappa2014transitory}. In this paper we have shown how techniques from the theory of stochastic-process limits, and more specifically heavy-traffic diffusion approximations, can be successfully employed to prove convergence results for the $\DG$ queue. In particular, we have proven that when suitably rescaled (according to a non-standard scaling) the queue length process converges in distribution to a reflected Brownian motion with downwards parabolic drift. Our result is a generalization of \cite{bet2014heavy}, where the arrival times are assumed to follow an exponential distribution. There, more demanding embedding and martingale techniques were used. Therefore, the techniques we have introduced offer a significant computational and conceptual advantage, allowing one to easily study other quantities of interest of the $\DG$ model other than the queue length process. As an example of the strength of our approach, we have proven a sample path diffusion Little's Law, which relates the virtual waiting time and the queue length processes. 
\section*{Acknowledgments}%
The author is very grateful to Debankur Mukherjee and Jori Selen for suggesting many improvements to the manuscript and numerous helpful discussions.

\bibliographystyle{abbrv}

\bibliography{library} % for Windows

%\bibliography{../Common_files/library.bib} % for Mac
\end{document}